\documentclass[12pt]{article}
\usepackage[utf8]{inputenc}
\usepackage{amsmath}
\usepackage{amssymb}
\usepackage{mathrsfs}
\usepackage{amsthm}
\usepackage{cite}
\usepackage{hyperref}
\usepackage{enumerate}
\usepackage[left=0.7in,right=0.7in,top=0.7in,bottom=0.7in]{geometry}
\usepackage{titlesec}
\titleformat*{\section}{\large\bfseries}
\newtheorem{theorem}{Theorem}[section]

\newtheorem{corollary}[theorem]{Corollary}

\newtheorem{example}[theorem]{Example}

\newtheorem{remark}[theorem]{Remark}
\numberwithin{equation}{section}
\title{Generalized Drazin-Meromorphic Invertible Operators and Browder's Type Theorems
}
\author{\large Anuradha Gupta$^1$ and Ankit Kumar$^2$\\ $^1$ {\small Department of Mathematics, Delhi College of Arts and Commerce,}\\ {\small University of Delhi, Netaji Nagar, New Delhi-110023, India.}\\{\small E-mail:dishna2@yahoo.in}\\{\small $^2$Department of Mathematics, University of Delhi, New Delhi-110007, India.}\\{\small E-mail:1995ankit13@gmail.com}}
\date{}
\begin{document}
\maketitle
\begin{abstract}
In this paper we give necessary and sufficient conditions for a bounded linear operator $T$ to be generalized Drazin-Riesz invertible or generalized Drazin-meromorphic invertible. Also, we study generalized Browder's theorem and generalized a-Browder's theorem by means of set of interior points of various parts of spectrum of $T$.
\end{abstract}

\section{Introduction and Preliminaries}
Throughout this paper, let $\mathbb{N}$ and $\mathbb{C}$ denote the set of natural numbers and complex numbers, respectively.  Let $B(X)$ denote the Banach algebra of all bounded linear operators acting on a complex Banach space $X$. For $T \in B(X)$,  we denote the null space of $T$, Range of $T$, spectrum of $T$ and adjoint of $T$ by $N(T)$, $R(T)$, $\sigma(T)$ and $T^*$, respectively.  For a subset $A$ of $\mathbb{C}$ the set of interior points of $A$, the set of accumulation points of $A$,  and the set of isolated points of $A$ are denoted by int$(A)$,  acc$(A)$ and iso$(A)$, respectively.  Let $\alpha(T)=$ dim $N(T)$ and $\beta(T)= \mbox{codim}\thinspace R(T)$ be the nullity of $T$ and deficiency of $T$, respectively. 
 A bounded linear operator $T$ is said to be bounded below if it is injective and $R(T)$ is closed. The \emph{approximate point} and \emph{surjective spectra} are defined by
  \begin{align*}
  \sigma_a(T)&:=\{\lambda \in \mathbb{C}: \lambda I-T \thinspace \mbox{is not bounded below} \},\\
  \sigma_s(T)&:=\{\lambda \in \mathbb{C}: \lambda I-T \thinspace \mbox{is not surjective} \}, \thinspace \mbox{respectively}.
  \end{align*}
  It is well known that $\sigma_a(T^*)=\sigma_s(T)$.
  
   An operator $T \in B(X)$ is called a upper semi-Fredholm operator if  if $\alpha(T) < \infty $ and $R(T)$ is closed.  An operator $T \in B(X)$ is called an lower semi-Fredholm operator  if $\beta(T) < \infty $.  The  class of all upper semi-Fredholm operators (lower semi-Fredholm operators, respectively) is denoted by $\phi_{+} (X)$ ($\phi_{-}(X)$, respectively). $T \in \phi_{+} (X)$ $\cup$ $\phi_{-}(X).$ For $T \in \phi_{+} (X)$ $\cup$ $\phi_{-}(X),$ the index of $T$ is defined by ind $(T$):= $\alpha(T)-\beta(T)$.  The class of all upper semi-Weyl operators (lower semi-Weyl operators, respectively) is defined by $W_{+} (X)=\{T\in\phi_{+} (X):$ ind $(T) \leq 0\}$  ($W_{-} (X)=\{T\in\phi_{-} (X):$ ind $(T) \geq 0\} $, respectively).  An operator $T \in B(X)$ is called Weyl if $T \in W_{+} (X)  \cap W_{-}(X)$. The \emph{upper semi-Weyl}, \emph{lower semi-Weyl} and  \emph{Weyl spectra}  are defined by 
 \begin{align*}
  \sigma_{uw}(T)&:=\{\lambda \in \mathbb{C}:\lambda I-T \thinspace \mbox{is not upper semi-Weyl}\},\\
 \sigma_{lw}(T)&:=\{\lambda \in \mathbb{C}:\lambda I-T \thinspace  \mbox{is not lower semi-Weyl}\},\\
 \sigma_{w}(T)&:=\{\lambda \in \mathbb{C}:\lambda I-T \thinspace  \mbox{is not Weyl}\}, \thinspace  \mbox{respectively}.
 \end{align*}
 
   For a bounded linear operator $T$, the ascent denoted by  $p(T)$ is the smallest non negative integer $p$ such that $N(T^{p})= N(T^{p+1})$. If no such integer exists we set $p(T)= \infty$.  For  a bounded linear operator $T$, the descent  denoted by $q(T)$ is the smallest non negative integer $q$ such that $R(T^q)=R(T^{q+1})$. If no such integer exists we set $q(T)= \infty$. By \cite[Theorem 1.20]{1} we know that if both $p(T)$ and $q(T)$ are finite, then $p(T)=q(T)$.
   
   An bounded linear operator $T$ is called Drazin invertible if there exist a positive integer $n$ and $S \in B(X)$ such that 
   $$ST=TS, \thinspace \thinspace STS=S \thinspace \mbox{and} \thinspace T^{n+1} S=T^n.$$ Also, by  \cite[Theorem 1.132]{1} we know that $T$ is Drazin invertible if and only if  $p(T)=q(T) < \infty$.   An operator $T \in B(X)$ is called right Drazin invertible  if $q(T) < \infty$ and $R(T^{q})$ is closed. An operator $T \in B(X)$ is called  left Drazin invertible  if $p(T) < \infty$ and $R(T^{p+1})$ is closed. An operator $T \in B(X)$ is called upper semi-Browder(lower semi-Browder, respectively)  if it is an upper semi-Fredholm(lower semi- Fredholm, respectively)  and $p(T) < \infty$($q(T) < \infty$, respectively).  We say that a bounded linear operator $T$ is  Browder if it is upper semi-Browder and lower semi-Browder. The \emph{upper semi-Browder}, \emph{lower semi-Browder} and \emph{Browder spectra} are defined by
  \begin{align*}
  \sigma_{ub}(T):& =\{\lambda \in \mathbb{C}: \lambda I -T \thinspace \mbox{is not upper semi-Browder} \},\\
  \sigma_{lb}(T):& =\{\lambda \in \mathbb{C}: \lambda I -T \thinspace \mbox{is not lower semi-Browder} \},\\
  \sigma_b(T):& =\{\lambda \in \mathbb{C}: \lambda I -T \thinspace \mbox{is not Browder} \}, \thinspace \mbox{respectively}.
  \end{align*} Clearly, every Browder operator is Drazin invertible.

An operator $T \in B(X)$ is called meromorphic if $\lambda I-T$ is Drazin invertible for all $\lambda \in \mathbb{C} \setminus \{0\}$ and  is called Riesz if $\lambda I-T$ is Browder for all $\lambda \in \mathbb{C} \setminus \{0\}$. Clearly, every Riesz operator is meromorphic. For $T \in B(X)$, a subspace $M$ of $X$ is said to be $T$-$invariant$ if $T(M) \subset M$. For a $T$-invariant subspace $M$ of $X$ we define $T_{M}:M \rightarrow M$ by $T_M (x)= T(x),x\in M$. We say $T$ is completely reduced by the pair $(M,N)$ and it is denoted by $(M,N) \in Red(T)$ if $M$ and $N$ are two closed $T$-invariant subspaces of $X$ such that $X=M \oplus N$.

The quasi-nilpotent part of $T$, defined as $$ H_0(T):=\{x \in X :\lim_{n \to \infty} \vert\vert T^n x \vert\vert ^{1/n} =0\}. $$
Clearly, $H_0(T)$ is a $T$-invariant subspace, generally not closed.
An operator $T \in B(X)$ is said to have the single-valued extension property (SVEP) at $\lambda_{0} \in \mathbb{C}$ if for every neighbourhood $U$ of $\lambda_{0}$ the only analytic function $f:U \rightarrow X$ which satisfies the equation $(\lambda I-T)f(\lambda)=0$ is the constant function $f=0$.  An operator $T \in B(X)$ is said to have SVEP if it has SVEP at every $lambda \in \mathbb{C}$. Let $E$ be the set of all points $\lambda \in \mathbb{C}$ such that $T$ does not have SVEP at $\lambda$, then $E$ is an open set contained in the  interior of $\sigma(T)$. Thus, if $T$ has SVEP at each point  of an open punctured disc $\mathbb{D} \setminus \{\lambda_{0}\}$ centered at $\lambda_{0}$, $T$ also has SVEP at $\lambda_{0}$. Also, 
$$p(\lambda I-T) < \infty \Rightarrow  T \thinspace \mbox{has SVEP at} \thinspace \lambda$$
and 
$$q(\lambda I-T) < \infty \Rightarrow  T^{*} \thinspace \mbox{has SVEP at} \thinspace  \lambda.$$

An operator $T \in B(X)$ is called Kato if $R(T)$ is closed and $N(T) \subset R(T^n)$ for every $n \in \mathbb{N}$. An operator $T \in B(X)$ is called nilpotent if $T^n =0$ for some $n \in \mathbb{N}$ and called quasi-nilpotent if $\vert \vert T^n \vert \vert^{\frac{1}{n}} \rightarrow 0$, i.e $\lambda I-T$ is invertible for all $\lambda \in \mathbb{C} \setminus \{0\}$. Clearly, every nilpotent operator is uasi-nilpotent. 

For $T \in B(X)$  and a non negative integer $n$, we define $T_{[n]}=T_{R(T^n)}$: $R(T^n)\rightarrow R(T^n)$. If for some non negative integer $n$ the range space $R(T^n)$ is closed and $T_{[n]}$ is  lower semi-Fredholm (an upper semi  Fredholm,  Fredholm, a lower semi Browder, an upper semi Browder, Browder, respectively) then $T$ is said to be lower semi B-Fredholm (an upper semi B-Fredholm, B-Fredholm, a lower semi B-Browder, an upper semi B-Browder, B-Browder, respectively). For a semi B-Fredholm operator $T$ (see \cite{3}), the index of $T$ is defined as index of  $T_{[n]}$. The  \emph{upper semi B-Browder}, \emph{lower semi B-Browder} and \emph{B-Browder spectra}  are defined by  
\begin{align*}
\sigma_{usbb}(T)&:=\{\lambda \in \mathbb{C}: \lambda I-T \thinspace \mbox{is not upper semi B-Browder}\},\\
\sigma_{lsbb}(T)&:=\{\lambda \in \mathbb{C}: \lambda I-T \thinspace \mbox{is not lower semi B-Browder}\},\\
\sigma_{bb}(T)&:=\{\lambda \in \mathbb{C}: \lambda I-T \thinspace \mbox{is not B-Browder}\}, \thinspace \mbox{respectively.}
\end{align*}
By \cite[Theorem 3.47]{1} an operator $T \in B(X)$ is left Drazin invertible (right Drazin invertible, Drazin invertible, respectively) if and only if $T$ is upper semi B-Browder (lower semi B-Browder, B-Browder, respectively).

An operator $T \in B(X)$ is called   an upper semi B-Weyl (a lower semi B-Weyl, respectively) if it is an upper semi B-Fredholm (a lower semi B-Fredholm, respectively)  having ind $(T)\leq 0$ (ind $(T) \geq 0$, respectively). An operator $T \in B(X)$ is called B-Weyl if it is B-Fredholm  and ind $(T)=0$. The \emph{upper semi B-Weyl}, \emph{lower semi B-Weyl} and  \emph{B-Weyl spectra} are defined by  
\begin{align*}
\sigma_{lsbw}(T)&:=\{\lambda \in \mathbb{C}: \lambda I-T \thinspace \mbox{is not upper semi B-Weyl}\},\\
\sigma_{usbw}(T)&:=\{\lambda \in \mathbb{C}: \lambda I-T \thinspace \mbox{is not lower semi B-Weyl}\},\\
\sigma_{bw}(T)&:=\{\lambda \in \mathbb{C}: \lambda I-T \thinspace \mbox{is not B-Weyl}\}, \thinspace \mbox{respectively.}
\end{align*}
By theorem \cite[Theorem 2.7]{3}) we know that a bounded linear operator $T$ is B-Fredholm (B-Weyl, respectively) if there exists $(M,N) \in Red(T)$ such that $T_M$ is Fredholm (Weyl, respectively) and $T_N$ is nilpotent. 

For $T \in B(X)$ consider the set 
$$\triangle (T):=\{n \in \mathbb{N}: m \geq n, m \in \mathbb{N} \mbox{ implies that } R(T^n) \cap N(T) \subset T^n(X) \cap N(T)\}. $$
The degree of stable iteration is defined by dis$(T)$:= inf $\triangle(T)$ whenever $\triangle(T)\neq\phi$. If $\triangle(T)=\phi$,  set dis$(T)= \infty$. 
An operator $T \in B(X)$ is said to be quasi-Fredholm of degree d if there exists a $d \in \mathbb{N}$ such that \\
(i) dis$(T)=d,$\\
(ii) $T^n (X)$ is a closed subspace of $X$ for each $n \geq d$,\\
(iii) $T(X)+$ $N(T^d)$ is a closed subspace of $X$.

An operator $T \in B(X)$ is said to admit a \textit{generalized kato decomposition} ($GKD$) if there exists a pair $(M,N) \in Red(T)$ such that $T_M$ is Kato and $T_N$ is quasi-nilpotent. If in the above definition, we assume $T_N$ to be nilpotent  then $T$ is said to be of Kato Type. (see \cite{13})  An operator is said to admit a \textit{generalized Kato-Riesz decomposition} ($GKRD$), if there exists a pair $(M,N) \in Red(T)$ such that $T_M$ is Kato and $T_N$ is Riesz. 

Recently, \v{Z}ivkovi\'{c}-Zlatanovi\'{c} and  Duggal \cite{14} introduced the notion of generalized Kato-meromorphic decomposition. An operator $T \in B(X)$ is said to admit a \emph{generalized Kato-meromorphic decomposition} ($GKMD$), if there exists a pair $(M,N) \in Red(T)$ such that $T_M$ is Kato and $T_N$ is meromorphic. \\

By theorem \cite[Theorem 1.132]{1} we know that  an operator $T \in B(X)$ is  Drazin invertible if and only if there exists $S \in B(X)$ such that $TS=ST$, $STS=S$ and $TST-T$ is nilpotent if and only if there exists  a pair $(M,N) \in Red(T)$ such that $T_M$ is invertible and $T_N$ is nilpotent.  Koliha \cite{4} generalized thE concept of Drazin-invertibility by replacing the third condition with $TST-T$ is quasi-nilpotent. An operator $T \in B(X)$ is said to be generalized Drazin invertible if there exist  a pair $(M,N) \in Red(T)$ such that $T_M$ is invertible and $T_N$ is quasi-nilpotent. Also, an operator $T \in B(X)$ is generalized drazin invertible if and only if $0 \notin$ acc$\sigma(T)$. (see \cite{5}) An operator $T \in B(X)$ is called generalized Drazin bounded below (surjective, respectively) if there exists a pair $(M,N) \in Red(T)$ such that $T_M$ is bounded below (surjective, respectively) and $T_N$ is quasi-nilpotent. The \emph{generalized Drazin bounded below}, \emph{generalized Drazin surjective} and \emph{generalized Drazin invertible spectra} are defined by
 \begin{align*}
\sigma_{gD\mathcal{J}}(T)&:=\{\lambda \in \mathbb{C}: \lambda I-T \thinspace \mbox{is not generalized Drazin bounded below}\},\\
\sigma_{gD\mathcal{Q}}(T)&:=\{\lambda \in \mathbb{C}: \lambda I-T \thinspace \mbox{is not generalized Drazin surjective}\},\\
\sigma_{gD}(T)&:=\{\lambda \in \mathbb{C}: \lambda I-T \thinspace \mbox{is not generalized Drazin invertible}\}, \thinspace \mbox{respectively.}
\end{align*}
 Recently, \v{Z}ivkovi\'{c}-Zlatanovi\'{c}\ and Cvetkovi\'{c} \cite{13}  introduced the notion of  generalized Drazin-Riesz invertible   operator by replacing the third condition with $TST-T$ is Riesz. They proved that an operator $T \in B(X)$ is generalized Drazin-Riesz invertible if and only if there exists  a pair $(M,N) \in Red(T)$ such that $T_M$ is invertible and $T_N$ is Riesz. Also, an operator $T \in B(X)$ is called generalized Drazin-Riesz bounded below (surjective, respectively) if there exists a pair $(M,N) \in Red(T)$ such that $T_M$ is bounded below (surjective, respectively) and $T_N$ is Riesz. The \emph{generalized Drazin-Riesz bounded below}, \emph{generalized Drazin-Riesz surjective} and \emph{generalized Drazin-Riesz invertible spectra} are defined by
 \begin{align*}
\sigma_{gDR\mathcal{J}}(T)&:=\{\lambda \in \mathbb{C}: \lambda I-T \thinspace \mbox{is not generalized Drazin-Riesz bounded below}\},\\
\sigma_{gDR\mathcal{Q}}(T)&:=\{\lambda \in \mathbb{C}: \lambda I-T \thinspace \mbox{is not generalized Drazin-Riesz surjective}\},\\
\sigma_{gDR}(T)&:=\{\lambda \in \mathbb{C}: \lambda I-T \thinspace \mbox{is not generalized Drazin-Riesz invertible}\}, \thinspace \mbox{respectively.}
\end{align*}
Also, they introduced the notion  of generalized Drazin-Riesz upper semi-Weyl, generalized Drazin-Riesz lower semi-Weyl and generalized Drazin-Riesz  Weyl operator. A bounded linear operator $T$ is called generalized Drazin-Riesz upper semi-Weyl (generalized Drazin-Riesz lower semi-Weyl, generalized Drazin-Riesz Weyl, respectively) if there exists $(M,N) \in Red(T)$ such that $T_M$ is   upper semi-Weyl(lower semi-Weyl, Weyl, respectively) and $T_N$ is Riesz. The \emph{generalized Drazin-Riesz upper  semi-Weyl}, \emph{generalized Drazin-Riesz lower semi-Weyl} and \emph{generalized Drazin-Riesz Weyl spectra} are defined by
\begin{align*}
\sigma_{gDRW_{+}}(T)&:=\{\lambda \in \mathbb{C}: \lambda I-T \thinspace \mbox{is not generalized Drazin-Riesz upper semi-Weyl}\},\\
\sigma_{gDRW_{-}}(T)&:=\{\lambda \in \mathbb{C}: \lambda I-T \thinspace \mbox{is not generalized Drazin-Riesz lower semi-Weyl}\},\\
\sigma_{gDRW}(T)&:=\{\lambda \in \mathbb{C}: \lambda I-T \thinspace \mbox{is not generalized Drazin-Riesz Weyl}\},
 \thinspace \mbox{respectively.}
\end{align*}
 Also, \v{Z}ivkovi\'{c}-Zlatanovi\'{c}   and Duggal   \cite{14} introduced the notion of  generalized Drazin-meromorphic invertible operator by replacing the third condition with $TST-T$ is meromorphic. They proved that the an operator $T \in B(X)$ is generalized Drazin-meromorphic invertible if and only if there exists  a pair $(M,N) \in Red(T)$ such that $T_M$ is invertible and $T_N$ is meromorphic. Also, an operator $T \in B(X)$ is said to be  generalized Drazin-meromorphic  bounded below  (surjective, respectively) if there exists a pair $(M,N) \in Red(T)$ such that $T_M$ is bounded below  (surjective, respectively) and $T_N$ is meromorphic. The \emph{generalized Drazin-meromorphic  bounded below}, \emph{generalized Drazin-meromorphic  surjective} and \emph{generalized Drazin-meromorphic invertible spectra} are defined by
 \begin{align*}
\sigma_{gDM\mathcal{J}}(T)&:=\{\lambda \in \mathbb{C}: \lambda I-T \thinspace \mbox{is not generalized Drazin-meromorphic bounded below}\},\\
\sigma_{gDM\mathcal{Q}}(T)&:=\{\lambda \in \mathbb{C}: \lambda I-T \thinspace \mbox{is not generalized Drazin-meromorphic surjective}\},\\
\sigma_{gDM}(T)&:=\{\lambda \in \mathbb{C}: \lambda I-T \thinspace \mbox{is not generalized Drazin-meromorphic invertible}\}, \thinspace \mbox{respectively.}
\end{align*}
Also, they introduced the notion  of generalized Drazin-meromorphic upper semi-Weyl, generalized Drazin-meromorphic lower semi-Weyl and generalized Drazin-meromorphic  Weyl operator. A bounded linear operator $T$ is called generalized Drazin-meromophic upper semi-Weyl (generalized Drazin-meromorphic lower semi-Weyl, generalized Drazin-meromorphic Weyl, respectively) if there exists $(M,N) \in Red(T)$ such that $T_M$ is   upper semi-Weyl(lower semi-Weyl, Weyl, respectively) and $T_N$ is meromorphic. The \emph{generalized Drazin-meromorphic upper  semi-Weyl}, \emph{generalized Drazin-meromorphic lower semi-Weyl} and \emph{generalized Drazin-meromorphic Weyl spectra} are defined by
\begin{align*}
\sigma_{gDMW_{+}}(T)&:=\{\lambda \in \mathbb{C}: 
\lambda I-T \thinspace \mbox{is not generalized Drazin-meromorphic upper semi-Weyl}\},\\
\sigma_{gDMW_{-}}(T)&:=\{\lambda \in \mathbb{C}: \lambda I-T \thinspace \mbox{is not generalized Drazin-meromorphic lower semi-Weyl}\},\\
\sigma_{gDMW}(T)&:=\{\lambda \in \mathbb{C}: \lambda I-T \thinspace \mbox{is not generalized Drazin-meromorphic Weyl}\},
 \thinspace \mbox{respectively.}
\end{align*}
Clearly,
\begin{align*}
 \sigma_{gDM\mathcal{J}} (T) \subset \sigma_{gDR\mathcal{J}}(T)  \subset \sigma_{gD\mathcal{J}}(T),\\
 \sigma_{gDM\mathcal{Q}}(T)  \subset \sigma_{gDR\mathcal{Q}}(T) \subset \sigma_{gD\mathcal{Q}}(T),\\
 \sigma_{gDM} \subset \sigma_{gDR}(T)\subset \sigma_{gD}(T).
\end{align*}
From \cite{13,14} we have
\begin{align*}
 \sigma_{gD*W_{+}} (T) \subset \sigma_{gD*\mathcal{J}}(T),\\
 \sigma_{gD*W_{-}}(T)  \subset \sigma_{gD*\mathcal{Q}}(T),\\
 \sigma_{gD*W}(T) \subset \sigma_{gD*}(T),
\end{align*}
where $*$ stands for Riesz or meromorphic operators.\\
Recently, Cvetkovi\'{c} and \v{Z}ivkovi\'{c}-Zlatanovi\'{c}\cite{6} obtained necessary and sufficient conditions for a bounded linear operator $T$ to be Drazin invertible or generalized Drazin invertible. Their conditions are of the form: $T$ satisfies some kind of decomposition property and $0$ is not an interior point of some part of spectrum of $T.$ Motivated by them we get similar kind of conditions for a bounded linear operator $T$ to be generalized Drazin-Riesz invertible or generalized Drazin-meromorphic invertible.

  \section{Main Results}
 \begin{theorem}\label{theorem1}
 Let $T \in B(X)$. Then 
 
 (i) $T$ is Drazin-invertible if and only if $T$ is upper semi B-Weyl and generalized Drazin-meromorphic surjective,
 
 (ii) $T$ is Drazin-invertible if and only if $T$ is  lower semi B-Weyl and generalized Drazin-meromorphic bounded below,
 
 (iii) $T$ is Browder if and only if $T$ is upper semi-Weyl and generalized Drazin-meromorphic surjective,
 
 (iv)  $T$ is Browder if and only if $T$ is lower semi-Weyl and generalized Drazin-meromorphic bounded below.
\end{theorem}  
\begin{proof}
(i) As $T$ is  generalized Drazin-meromorphic surjective, by \cite[Theorem 7]{14} we have $T^{*}$ has SVEP at 0. By \cite[Theorem 1.116]{1} we know that every upper semi B-Weyl is quasi-Fredholm. By \cite[Theorem 2.90]{20} we get $q(\lambda I-T) < \infty$. Thus, by \cite[Theorem 1.143]{1} we have $T$ is Drazin-invertible. The converse is obvious.  

(ii)  As $T$ is  generalized Drazin-meromorphic bounded below, by \cite[Theorem 6]{14} we have $T$ has SVEP at 0. By \cite[Theorem 1.116]{1} we know that every lower semi B-Weyl is quasi-Fredholm. By \cite[Theorem 2.90]{20} we get $p(\lambda I-T) < \infty$. Thus, by \cite[Theorem 1.143]{1} we have $T$ is Drazin-invertible. The converse is obvious.  

(iii) By (i) we get $T$ is Drazin invertible. Since $T$ is upper semi-Weyl, $T$ is Browder.

(iv)  It is similar to (iii).
\end{proof}
\begin{theorem}\label{theorem2}
Let $T \in B(X)$, then 

(i) iso$\sigma_{usbw}(T) \subset$ iso$\sigma_{bb}(T)$  $\cup$  int$\sigma_{gDM\mathcal{Q}}(T)$,

(ii)  iso$\sigma_{lsbw}(T) \subset$ iso$\sigma_{bb}(T)$  $\cup$  int$\sigma_{gDM\mathcal{J}}(T)$,

(iii) $\sigma_{bb}(T)= \sigma_{usbw}(T)$  $\cup$ int$\sigma_{gDM\mathcal{Q}}(T)$,

(iv) $\sigma_{bb}(T)= \sigma_{lsbw}(T)$  $\cup$ int$\sigma_{gDM\mathcal{J}}(T)$,

(v) iso$\sigma_{uw}(T) \subset$ iso$\sigma_{b}(T)$  $\cup$  int$\sigma_{gDM\mathcal{Q}}(T)$,

(vi)  iso$\sigma_{lw}(T) \subset$ iso$\sigma_{b}(T)$  $\cup$  int$\sigma_{gDM\mathcal{J}}(T)$,

(vii) $\sigma_{b}(T)= \sigma_{uw}(T)$  $\cup$ int$\sigma_{gDM\mathcal{Q}}(T)$,

(viii) $\sigma_{b}(T)= \sigma_{lw}(T)$  $\cup$ int$\sigma_{gDM\mathcal{J}}(T)$.
\end{theorem}
\begin{proof}
(i) Let $\lambda_{0} \in $ iso$\sigma_{usbw}(T) \setminus$ int$\sigma_{gDM\mathcal{Q}}(T)$. Then there exists a sequence \{$\lambda_{n}$ \}$\rightarrow \lambda_{0}$ as $n \rightarrow \infty$, such that $\lambda_n I-T$ is both upper semi B-Weyl and  generalized Drazin-meromorphic surjective for all $n$. By theorem \ref{theorem1} we get  $\lambda_n I-T$ is Drazin-invertible. This gives $\lambda _{0} \in \partial \sigma_{bb}(T)$. By \cite [Theorem 3.8]{12} we know that $\partial \sigma_{bb}(T) \subset \sigma_{usbw}(T) \subset \sigma_{bb}(T).$ Thus, by \cite [Theorem 4]{6} we get $\lambda_{0} \in$ iso$\sigma_{bb}(T)$.
(ii)  Let $\lambda_{0} \in $ iso$\sigma_{lsbw}(T) \setminus$ int$\sigma_{gDM\mathcal{J}}(T)$. By Theorem \ref{theorem1} we know that $T$ is both lower semi B-Weyl and  generalized meromorphic bounded below, then $T$ is Drazin invetible. Using the similar argument as in (i) we get  $\lambda _{0} \in \partial \sigma_{bb}(T)$. Using \cite [Theorem 3.8]{12} and \cite[Theorem 4]{6}  we get $\lambda _{0} \in \sigma_{bb}(T).$

(iii) Since every Drazin-invetible is upper semi B-Weyl and  generalized Drazin-meromorphic sujective, we have $ \sigma_{usbw}(T)$ $\cup$ int$\sigma_{gDM\mathcal{Q}}(T) \subset \sigma_{bb}(T)$. Let $\lambda_{0} \in 
\sigma_{bb}(T)$ and $\lambda_{0} \notin \sigma_{usbw}(T)$ $\cup$ int$\sigma_{gDM\mathcal{Q}}(T)$. Then using the similar argument as in (i)  we get $\lambda_{0} \in \partial \sigma_{bb}(T) \subset \sigma_{usbw}(T)$, which is a contradiction.

(iv) Using similar arguments as in part (iii) we get desired result.

By \cite[corollary 2.5]{18} we know that  $\partial \sigma_{b}(T) \subset \sigma_{uw}(T)$ and  $\partial \sigma_{b}(T) \subset \sigma_{lw}(T)$. Therefore, (v),(vi),(vii) and (viii) can be proved in similar way as (i),(ii),(iii) and (iv), respectively.
\end{proof}
\begin{corollary}\label{corollary1}
Let $T \in B(X)$. Then 

(i) $\sigma_{bb}(T)=\sigma_{usbw}(T)$ $\cup \mbox{int} \sigma_{*}(T)$, where $\sigma_{*}=\sigma_{bb}$ or $ \sigma_{lsbb}$ or $\sigma_{gD}$ or $\sigma_{gD\mathcal{Q}}$ or $\sigma_{gDR}$ or $\sigma_{gDR\mathcal{Q}}$ or $\sigma_{gDM}$,

(ii) $\sigma_{bb}(T)=\sigma_{bw}(T)$ $\cup \mbox{int} \sigma_{*}(T)$, where $\sigma_{*}=\sigma_{bb}$ or $ \sigma_{lsbb}$ or $\sigma_{gD}$ or $\sigma_{gD\mathcal{Q}}$ or $\sigma_{gDR}$ or $\sigma_{gDR\mathcal{Q}}$ or $\sigma_{gDM}$,

(iii) $\sigma_{b}(T)=\sigma_{uw}(T)$ $\cup \mbox{int} \sigma_{*}(T)$, where $\sigma_{*}=\sigma_{b}$ or $\sigma_{lb}$ or $\sigma_{bb}$ or $ \sigma_{lsbb}$ or $\sigma_{gD}$ or $\sigma_{gD\mathcal{Q}}$ or $\sigma_{gDR}$ or $\sigma_{gDR\mathcal{Q}}$ or $\sigma_{gDM}$,

(iv) $\sigma_{b}(T)=\sigma_{w}(T)$ $\cup \mbox{int} \sigma_{*}(T)$, where $\sigma_{*}=\sigma_{b}$ or $\sigma_{lb}$ or $\sigma_{bb}$ or $ \sigma_{lsbb}$ or $\sigma_{gD}$ or $\sigma_{gD\mathcal{Q}}$ or $\sigma_{gDR}$ or $\sigma_{gDR\mathcal{Q}}$ or $\sigma_{gDM}$.
\end{corollary}
\begin{proof}
(i) Since $\sigma_{gDM\mathcal{Q}}(T) \subset \sigma_{*}(T)$,  where $\sigma_{*}=\sigma_{bb}$ or $ \sigma_{lsbb}$ or $\sigma_{gD}$ or $\sigma_{gD\mathcal{Q}}$ or $\sigma_{gDR}$ or $\sigma_{gDR\mathcal{Q}}$ or $\sigma_{gDM}$ by Theorem \ref{theorem2} we get $\sigma_{bb}(T) \subset  \sigma_{usbw}(T)$ $\cup $ int$\sigma_{*}(T)$. The converse is obvious.

(ii) Since $\sigma_{usbw}(T) \subset \sigma_{bw}(T) \subset \sigma_{bb}(T)$. Using (i) we get the desired result.

(iii) and (iv) can be proved in similar way.
\end{proof}
\begin{corollary}\label{corollary2}
Let $T \in B(X)$. Then 

(i) $\sigma_{bb}(T)=\sigma_{lsbw}(T)$ $\cup \mbox{int} \sigma_{*}(T)$, where $\sigma_{*}=\sigma_{bb}$ or $ \sigma_{usbb}$ or $\sigma_{gD}$ or $\sigma_{gD\mathcal{J}}$ or $\sigma_{gDR}$ or $\sigma_{gDR\mathcal{J}}$ or $\sigma_{gDM}$,

(ii) $\sigma_{bb}(T)=\sigma_{bw}(T)$ $\cup \mbox{int} \sigma_{*}(T)$, where $\sigma_{*}=\sigma_{bb}$ or $ \sigma_{usbb}$ or $\sigma_{gD}$ or $\sigma_{gD\mathcal{J}}$ or $\sigma_{gDR}$ or $\sigma_{gDR\mathcal{J}}$ or $\sigma_{gDM}$,

(iii) $\sigma_{b}(T)=\sigma_{lw}(T)$ $\cup \mbox{int} \sigma_{*}(T)$, where $\sigma_{*}=\sigma_b $ or $\sigma_{ub}$ or $ \sigma_{bb}$ or $ \sigma_{usbb}$ or $\sigma_{gD}$ or $\sigma_{gD\mathcal{J}}$ or $\sigma_{gDR}$ or $\sigma_{gDR\mathcal{J}}$ or $\sigma_{gDM}$,

(iv) $\sigma_{bb}(T)=\sigma_{bw}(T)$ $\cup \mbox{int} \sigma_{*}(T)$, where $\sigma_{*}=\sigma_b $ or $\sigma_{ub}$ or $ \sigma_{bb}$ or $ \sigma_{usbb}$ or $\sigma_{gD}$ or $\sigma_{gD\mathcal{J}}$ or $\sigma_{gDR}$ or $\sigma_{gDR\mathcal{J}}$ or $\sigma_{gDM}$.
\end{corollary}
\begin{proof}
(i) Since $\sigma_{gDM\mathcal{J}}(T) \subset \sigma_{*}(T)$, by Theorem \ref{theorem2} we get $\sigma_{bb}(T) \subset  \sigma_{lsbw}(T)$  $\cup $ int$\sigma_{*}(T)$. The converse is obvious.

(ii) Since $\sigma_{lsbw}(T) \subset \sigma_{bw}(T) \subset \sigma_{bb}(T)$, we get the desired result.

(iii) and (iv) can be done in similar way.
\end{proof}

\begin{remark}\label{remark1}
\emph{Let $T \in B(X)$ such that $T$ is meromorphic operator with infinite spectrum. Thus, $\sigma_{bb}(T)=$int$\sigma_{bb}(T)=\{0\}.$ Therefore, by Corollary  \ref{corollary2} we get $\sigma_{bw}(T)=\sigma_{usbw}(T)=\sigma_{lsbw}(T)=\{0\}.$}
\end{remark}

\begin{theorem}\label{theoem3}
Let $T \in B(X)$. Then the following statements are equivalent:

(i) $T$ is generalized Drazin-meromorphic invertible,

(ii) $0 \notin$ int$\sigma_{*}(T)$ and there exists $(M,N) \in Red(T)$ such that $T_M$ is upper semi-Weyl and $T_N$ is meromorphic, where $\sigma_{*}=\sigma_{bb}$ or $ \sigma_{lsbb}$ or $\sigma_{gD}$ or $\sigma_{gD\mathcal{Q}}$ or $\sigma_{gDR}$ or $\sigma_{gDR\mathcal{Q}}$ or $\sigma_{gDM}$ or $\sigma_{gDM\mathcal{Q}}$,

(iii)  $0 \notin$ int$\sigma_{*}(T)$ and there exists $(M,N) \in Red(T)$ such that $T_M$ is lower semi- Weyl and $T_N$ is meromorphic,  where $\sigma_{*}=\sigma_{bb}$ or $ \sigma_{usbb}$ or $\sigma_{gD}$ or $\sigma_{gD\mathcal{J}}$ or $\sigma_{gDR}$ or $\sigma_{gDR\mathcal{J}}$ or $\sigma_{gDM}$ or $\sigma_{gDM\mathcal{J}}$.
\end{theorem}
\begin{proof}
(i) $\Rightarrow$ (ii) It suffices to prove that $0 \notin$ int$\sigma_{gDM\mathcal{Q}}(T).$ As $T$ is generalized Drazin-meromorphic invertible, by \cite[theorem 5]{14} we have $0 \notin$ int$\sigma_{bb}(T)$. As  $\sigma_{gDM\mathcal{Q}}(T) \subset \sigma_{bb}(T)$, $0 \notin$ int$\sigma_{gDM\mathcal{Q}}(T)$.

(ii) $\Rightarrow$ (i) Since there exists $(M,N) \in Red(T)$ such that $T_M$ is upper semi-Weyl and $T_N$ is meomorphic, by \cite[Theorem 1]{14} we get  T admits a GKMD and $0 \notin$ acc$\sigma_{usbw}(T)$. Let $0\notin \sigma_{usbw}(T)$.  As $0 \notin$ int$\sigma_{gDM\mathcal{Q}}(T)$,  by  Theorem  \ref{theorem2} we get $0 \notin \sigma_{bb}(T)$ which implies that $0 \notin$ acc$\sigma_{bb}(T)$. Let  $0\in \sigma_{usbw}(T)$, $0 \in$  iso$\sigma_{usbw}(T)$ which implies that $0 \in$  iso$\sigma_{bb}(T)$. Thus, in both cases  by \cite[Theorem 5]{14} $T$ is generalized Drazin-meromorphic invertible.

Similarly using Theorem \ref{theorem2} and \cite[Theorem1]{14} we can prove (i)$\Leftrightarrow$ (iii) .
\end{proof}
\begin{corollary}\label{corollary3}
Let $T \in B(X)$. Then

(i) $\sigma_{gDM}(T)=\sigma_{gDMW_{+}}(T)$ $ \cup$ int$\sigma_{*}(T),$ where $\sigma_{*}=\sigma_{bb}$ or $ \sigma_{lsbb}$ or $\sigma_{gD}$ or $\sigma_{gD\mathcal{Q}}$ or $\sigma_{gDR}$ or $\sigma_{gDR\mathcal{Q}}$ or $\sigma_{gDM}$ or $\sigma_{gDM\mathcal{Q}}$,

(ii) $\sigma_{gDM}(T)=\sigma_{gDMW_{-}}(T)$ $ \cup$ int$\sigma_{*}(T),$ where $\sigma_{*}=\sigma_{bb}$ or $ \sigma_{usbb}$ or $\sigma_{gD}$ or $\sigma_{gD\mathcal{J}}$ or $\sigma_{gDR}$ or $\sigma_{gDR\mathcal{J}}$ or $\sigma_{gDM}$ or $\sigma_{gDM\mathcal{J}}$,

(iii)$\sigma_{gDM}(T)=\sigma_{gDMW}(T)$ $ \cup$ int$\sigma_{*}(T),$ where $\sigma_{*}=\sigma_{bb}$ or $ \sigma_{lsbb}$ or  $ \sigma_{usbb}$ or $\sigma_{gD}$ or $\sigma_{gD\mathcal{Q}}$ or  $ \sigma_{llsbb}$ or $\sigma_{gDR}$ or $\sigma_{gDR\mathcal{Q}}$ or  $ \sigma_{gDR\mathcal{J}}$ or  $ \sigma_{gDM\mathcal{Q}}$ or $\sigma_{gDM}$,

(iv)int$\sigma_{*}(T) =$ int$\sigma_{gDM}(T)$, where $\sigma_{*}= \sigma_{bb}$ or $ \sigma_{gD}$ or  $ \sigma_{gDR}$.
\end{corollary}
\begin{proof}
Using Theorem \ref{theoem3} we can easily deduce (i) and (ii).\\
(iii) Since $\sigma_{gDMW_{+}}(T)$ and $\sigma_{gDMW_{-}}(T)$ both are subsets of $\sigma_{gDMW}(T)$. Also, $\sigma_{gDMW_{+}}(T) \subset \sigma_{gDM}(T)$. Using part (i) and (ii) we get the desired result.

(iv) By (i) and (ii) we get int$\sigma_{*}(T) \subset \sigma_{gDM}(T)$ which implies that int$\sigma_{*}(T) \subset$ int$\sigma_{gDM}(T)$,  where $\sigma_{*}= \sigma_{bb}$ or $ \sigma_{gD}$ or  $ \sigma_{gDR}$. The  converse is obvious.
\end{proof}
Recently, Aiena and Triolo \cite[Theorem 2.8]{19} proved that if $H_0(T)$ is closed then $0 \in$ iso$\sigma_s (T)$ if and only if $0 \in$ iso$\sigma(T).$ Also, Cvetkovi\'{c} and \v{Z}ivkovi\'{c}-Zlatanovi\'{c} \cite[Corollary 2.3]{6} obtained condition for $0$ to be an isolated point of both $\sigma_{lw}(T)$ and $\sigma(T).$ 
\begin{theorem}\label{theorem4}
Let $T \in B(X)$ such that $T$ admits a GKD and $H_0(T)$ is closed. Let $0 \in \sigma_{lsbw}(T).$ Then
$$0 \in \mbox{iso} \sigma_{lsbw}(T) \Leftrightarrow 0 \in \mbox{iso} \sigma(T)$$
\end{theorem}
\begin{proof}
Let $0 \in$ iso$\sigma(T)$. Since $\sigma_{lsbw}(T) \subset \sigma(T)$ and $0 \in \sigma_{lsbw}(T)$, we get $0 \in \mbox{iso} \sigma_{lsbw}(T).$ Conversely, suppose that $0 \in \mbox{iso} \sigma_{lsbw}(T)$. Since $T$ admits a GKD, $T$ admits a GKMD. Thus, by \cite[Theorem 1]{14} there exists $(M,N) \in Red(T)$ such that $T_M$ is lower semi-Weyl and $T_N$ is meromorphic.  Also by \cite[Theorem 3.5]{15} and \cite[Theorem 2.9]{1} we have $0 \notin$ int$\sigma_{a}(T)$ which implies that $0 \notin$ int$\sigma_{gDM\mathcal{J}}(T)$. Thus by Theorem \ref{theorem2} we get T is generalized Drazin-meromorphic invertible. Therefore, by \cite [Theorem 5]{14} $0 \notin$ int$\sigma(T)$. As $T$ admits GKD, by \cite[Theorem 3.9]{5} we get $T$ is generalized Drazin-invertible which implies that $0 \in$ iso$\sigma(T)$.
\end{proof}
We see that if we drop assumption that $T$ admits GKD or $H_0(T)$ is closed then Theorem \ref{theorem4} need not be true.
\begin{example}
\emph{Let $L$ be the classical unilateral left shift on $l^2(\mathbb{N})$. Define $S:l^2(\mathbb{N}) \rightarrow l^2(\mathbb{N})$ by $$S(x_0,x_1,x_2,\ldots)= (\frac{x_1}{2},\frac{x_2}{3},\frac{x_3}{4},\ldots).$$ Then $S$ is a quasi-nilpotent. Define $T=L \oplus S^*$ on $l^2(\mathbb{N}) \oplus l^2(\mathbb{N})$. Then $\sigma(T)=D$ and $\sigma_s (T)= S^1$ $ \cup \{0\},$ where $D$ denotes the unit closed disc and $S^1$ denotes the unit circle. Since $L$ is surjective and $S^*$ is quasi-nilpotent, $T$ admits a GKD. By \cite[Lemma 2.81]{1} $ H_0(T)$ is not closed. As $q(S^*)$is infinite, $0\in \sigma_{lsbb}(T).$ By \cite[Theorem 3.50]{1} we know that $\sigma_{lsbb}(T)=\sigma_{lsbw}(T)$ $ \cup$ acc$\sigma_s(T).$ Therefore, $0 \in \sigma_{lsbw}(T)$ which implies that $0 \in $ iso$\sigma_{lsbw}(T)$. On the other hand $0\in$ acc$\sigma(T).$}
\end{example}
\begin{example}
\emph{	\cite[Example 2.4]{6} Let $H$ be a complex separable Hilbert space with scalar product(,) and let \{$e_n:n \in \mathbb{N}$\} be a total orthonormal system in $H$. Let $\{\lambda_n\}$ be a sequence of complex numbers such that $\{\vert \lambda_n \vert\}$ is a decreasing sequence converging to zero then we define $$T(x)=\sum \lambda_n (x,e_n) e_n  \mbox{ for all $x \in H$}.$$ Then $T$ is a compact normal operator with $\sigma(T)=$\{$\lambda_n :n \in \mathbb{N}$\} $\cup$ \{$0$\}. As every compact operator is Riesz, by \cite[Remark 5.3]{5} $T$ does not admit GKD. As T is normal, $H_0(T)=$ $ N(T).$ Therefore, $H_0(T)$ is closed. As $\sigma(T)$ is infinite, $\sigma_{bb}(T)=\{0\}$. By Remark \ref{remark1} we get $\sigma_{lsbw}(T)=\{0\}.$ Thus, $0 \in $ iso$\sigma_{lsbw}(T)$ but $0 \in$ acc$\sigma(T).$}
\end{example}
\begin{theorem}\label{theorem7}

Let $T \in B(X)$. Then the following statements are equivalent:

(i) $T$ is generalized Drazin-Riesz invertible,

(ii) $0 \notin$ int$\sigma_{*}(T)$ and there exists $(M,N) \in Red(T)$ such that $T_M$ is upper semi-Weyl and $T_N$ is Riesz,  where $\sigma_{*}=\sigma_{b}$ or $\sigma_{lb}$ or $\sigma_{bb}$ or $ \sigma_{lsbb}$ or $\sigma_{gD}$ or $\sigma_{gD\mathcal{Q}}$ or $\sigma_{gDR}$ or $\sigma_{gDR\mathcal{Q}}$ or $\sigma_{gDM}$,

(iii)  $0 \notin$ int$\sigma_{*}(T)$ and there exists $(M,N) \in Red(T)$ such that $T_M$ is lower semi- Weyl and $T_N$ is Riesz,  where $\sigma_{*}=\sigma_b (T)$ or $\sigma_{ub}(T)$ or $ \sigma_{bb}$ or $ \sigma_{usbb}$ or $\sigma_{gD}$ or $\sigma_{gD\mathcal{J}}$ or $\sigma_{gDR}$ or $\sigma_{gDR\mathcal{J}}$ or $\sigma_{gDM}$.
\end{theorem}
\begin{proof}
(i) $\Rightarrow$ (ii) It suffices to prove that $0 \notin$ int$\sigma_{*}(T).$ As $T$ is generalized Drazin-Riesz invertible, by \cite[theorem 2.3]{13} we have $0 \notin$ int$\sigma_{b}(T)$. As  $\sigma_{*}(T) \subset \sigma_{b}(T)$, we get $0 \notin$ int$\sigma_{*}(T)$, where $\sigma_{*}=\sigma_{b}$ or $\sigma_{lb}$ or $\sigma_{bb}$ or $ \sigma_{lsbb}$ or $\sigma_{gD}$ or $\sigma_{gD\mathcal{Q}}$ or $\sigma_{gDR}$ or $\sigma_{gDR\mathcal{Q}}$ or $\sigma_{gDM}$.

(ii) $\Rightarrow$ (i) Since there exists $(M,N) \in Red(T)$ such that $T_M$ is upper semi-Weyl and $T_N$ is Riesz, by \cite[Theorem 2.6]{13} we get  T admits a GKRD and $0 \notin$ acc$\sigma_{uw}(T)$. Let $0 \notin \sigma_{uw}(T)$.  As $0 \notin$ int$\sigma_{*}(T)$,  by Corollary    \ref{corollary1}  we get $0 \notin \sigma_{b}(T)$ which implies that $0 \notin$ acc$\sigma_{b}(T)$. Let  $0\in \sigma_{uw}(T)$, $0 \in$  iso$\sigma_{uw}(T)$. Since $\partial \sigma_{b}(T) \subset \sigma_{uw}(T)$,  $0 \in$  iso$\sigma_{b}(T)$. Thus in both the cases  by \cite[Theorem 2.3]{13} $T$ is generalized Drazin-Riesz invertible.

Similarly, we can prove (i) $\Leftrightarrow$ (iii).
\end{proof}
\begin{corollary}\label{corollary4}
Let $T \in B(X)$. Then

(i)$\sigma_{gDR}(T)=\sigma_{gDMW_{+}}(T)$ $\cup$ int$\sigma_{*}(T),$ where $\sigma_{*}=\sigma_{b}$ or $\sigma_{lb}$ or $\sigma_{bb}$ or $ \sigma_{lsbb}$ or $\sigma_{gD}$ or $\sigma_{gD\mathcal{Q}}$ or $\sigma_{gDR}$ or $\sigma_{gDR\mathcal{Q}}$ or $\sigma_{gDM}$ or $\sigma_{gDM\mathcal{Q}}$,

(ii)$\sigma_{gDR}(T)=\sigma_{gDMW_{-}}(T)$ $ \cup$ int$\sigma_{*}(T),$ where $\sigma_{*}=\sigma_b$ or $\sigma_{ub}$ or $\sigma_{bb}$ or $ \sigma_{usbb}$ or $\sigma_{gD}$ or $\sigma_{gD\mathcal{J}}$ or $\sigma_{gDR}$ or $\sigma_{gDR\mathcal{J}}$ or $\sigma_{gDM}$ or $\sigma_{gDM\mathcal{J}}$,

(iii)$\sigma_{gDR}(T)=\sigma_{gDRW}(T)$ $ \cup$ int$\sigma_{*}(T),$ where $\sigma_{*}=\sigma_b$ or $\sigma_{lb}$ or $ \sigma_{ub}$ or $\sigma_{bb}$ or $ \sigma_{lsbb}$ or  $ \sigma_{usbb}$ or $\sigma_{gD}$ or $\sigma_{gD\mathcal{Q}}$ or  $ \sigma_{lsbb}$ or $\sigma_{gDR}$ or $\sigma_{gDR\mathcal{Q}}$ or  $ \sigma_{gDR\mathcal{J}}$ or  $ \sigma_{gDM\mathcal{Q}}$ or $\sigma_{gDM}$,

(iv)int $\sigma_{*}(T)=$ int$\sigma_{gDR}(T)$, where $\sigma_{*}=\sigma_{gD}$ or $\sigma_{bb}$ or $\sigma_{b}.$
\end{corollary}
\begin{proof}
Using Theorem \ref{theorem7} we  get (i) and (ii).

(iii) As $\sigma_{gDRW_{+}}(T)$ and $\sigma_{gDMW_{-}}(T)$ both are subsets of $\sigma_{gDRW}(T)$. Also, $\sigma_{gDRW}(T) \subset \sigma_{gDR}(T)$. Using (i) and (ii) we get the desired result. 

(iv) By (i) and (ii) we get int$\sigma_{*}(T) \subset \sigma_{gDR}(T)$ which implies that int$\sigma_{*}(T) \subset$ int$\sigma_{gDR}(T)$,  where $\sigma_{*}= \sigma_{b}$ or $ \sigma_{bb}$ or  $ \sigma_{gD}$. The converse is obvious.

By Corollary \ref{corollary3} and Corollary     \ref{corollary4} we get int$\sigma_{b}(T)=$ int$\sigma_{bb}(T)=$ int$\sigma_{gD}(T)=$ int$\sigma_{gDR}(T)=$ int$\sigma_{gDM}(T).$
\end{proof}
\section{Browder's type theorems}
Recall that a bounded linear operator $T$ satisfies Browder's theorem if $\sigma_b (T)=\sigma_w (T)$, generalized Browder's theorem if $\sigma_{bb}(T)=\sigma_{bw}(T)$, a-Browder's theorem if $\sigma_{uw}(T)=\sigma_{ub}(T)$ and generalized a-Browder's theorem if $\sigma_{usbb}(T)=\sigma_{usbw}(T)$.

\begin{theorem}\label{theorem9}
Let $T \in B(X)$. Then the following statements are equivalent:

(i) generalized Browder's theorem holds for $T$,

(ii) Browder's theorem holds for $T$,

(iii) int$\sigma_{*}(T) \subset \sigma_{bw}(T),$  where $\sigma_{*}=\sigma_{bb}$ or $ \sigma_{lsbb}$ or  $ \sigma_{usbb}$ or $\sigma_{gD}$ or $\sigma_{gD\mathcal{Q}}$ or  $ \sigma_{gD\mathcal{J}}$ or $\sigma_{gDR}$ or $\sigma_{gDR\mathcal{Q}}$ or  $ \sigma_{gDR\mathcal{J}}$ or  $ \sigma_{gDM\mathcal{Q}}$ or $\sigma_{gDM}$,

(iv) int$\sigma_{*}(T) \subset \sigma_{gDMW}(T),$  where $\sigma_{*}=\sigma_{bb}$ or $ \sigma_{lsbb}$ or  $ \sigma_{usbb}$ or $\sigma_{gD}$ or $\sigma_{gD\mathcal{Q}}$ or  $ \sigma_{gD\mathcal{J}}$ or $\sigma_{gDR}$ or $\sigma_{gDR\mathcal{Q}}$ or  $ \sigma_{gDR\mathcal{J}}$ or  $ \sigma_{gD\mathcal{Q}M}$ or $\sigma_{gDM}$,

(v) int$\sigma_{*}(T) \subset \sigma_{w}(T),$  where $\sigma_{*}=\sigma_{b}$ or $ \sigma_{ub}$ or  $ \sigma_{lb}$ or $\sigma_{bb}$ or $ \sigma_{lsbb}$ or  $ \sigma_{usbb}$ or $\sigma_{gD}$ or $\sigma_{gD\mathcal{Q}}$ or  $ \sigma_{gD\mathcal{J}}$ or $\sigma_{gDR}$ or $\sigma_{gDR\mathcal{Q}}$ or  $ \sigma_{gDR\mathcal{J}}$ or  $ \sigma_{gDM\mathcal{Q}}$ or $\sigma_{gDM}$,

(vi) int$\sigma_{*}(T) \subset \sigma_{gDRW}(T),$  where $\sigma_{*}=\sigma_{b}$ or $ \sigma_{ub}$ or  $ \sigma_{lb}$ or $\sigma_{bb}$ or $ \sigma_{lsbb}$ or  $ \sigma_{usbb}$ or $\sigma_{gD}$ or $\sigma_{gD\mathcal{Q}}$ or  $ \sigma_{gD\mathcal{J}}$ or $\sigma_{gDR}$ or $\sigma_{gDR\mathcal{Q}}$ or  $ \sigma_{gDR\mathcal{J}}$ or  $ \sigma_{gDM\mathcal{Q}}$ or $\sigma_{gDM}$.
\end{theorem}
\begin{proof}
(i)  $\Leftrightarrow$ (ii) By \cite[Theorem 2.1]{16} we know that $T \in B(X)$ satisfies Browder's theorem if and only if it satisfies generalized Browder's theorem.

(i) $\Leftrightarrow$ (iii) By Corollary \ref{corollary1} and Corollary \ref{corollary2} we know that $\sigma_{bb}(T)=\sigma_{bw}(T)$ if and only if int$\sigma_{*}(T) \subset \sigma_{bw}(T),$  where $\sigma_{*}=\sigma_{bb}$ or $ \sigma_{lsbb}$ or  $ \sigma_{usbb}$ or $\sigma_{gD}$ or $\sigma_{gD\mathcal{Q}}$ or  $ \sigma_{llsbb}$ or $\sigma_{gDR}$ or $\sigma_{gDR\mathcal{Q}}$ or  $ \sigma_{gDR\mathcal{J}}$ or  $ \sigma_{gDM\mathcal{Q}}$ or $\sigma_{gDM}$.

(i) $\Leftrightarrow$ (iv) By Corollary \ref{corollary3} we know that  $\sigma_{gDM}(T)=\sigma_{gDMW}(T)$ if and only if int$\sigma_{*}(T) \subset \sigma_{bw}(T),$  where $\sigma_{*}=\sigma_{bb}$ or $ \sigma_{lsbb}$ or  $ \sigma_{usbb}$ or $\sigma_{gD}$ or $\sigma_{gD\mathcal{Q}}$ or  $ \sigma_{llsbb}$ or $\sigma_{gDR}$ or $\sigma_{gDR\mathcal{Q}}$ or  $ \sigma_{gDR\mathcal{J}}$ or  $ \sigma_{gDM\mathcal{Q}}$ or $\sigma_{gDM}$.  by \cite [Theorem 2.8]{7} we get $\sigma_{gDM}(T)=\sigma_{gDMW}(T)$ if and only if $T$ satisfies  generalized Browder's theorem.

(ii) $\Leftrightarrow$ (v) By Corollary \ref{corollary1} and Corollary \ref{corollary2} we know that $\sigma_{b}(T)=\sigma_{w}(T)$ if and only if int$\sigma_{*}(T) \subset \sigma_{w}(T),$  where $\sigma_{*}=\sigma_{b}$ or $ \sigma_{ub}$ or  $ \sigma_{lb}$ or $\sigma_{bb}$ or $ \sigma_{lsbb}$ or  $ \sigma_{usbb}$ or $\sigma_{gD}$ or $\sigma_{gD\mathcal{Q}}$ or  $ \sigma_{llsbb}$ or $\sigma_{gDR}$ or $\sigma_{gDR\mathcal{Q}}$ or  $ \sigma_{gDR\mathcal{J}}$ or  $ \sigma_{gDM\mathcal{Q}}$ or $\sigma_{gDM}$.

(ii) $\Leftrightarrow$ (vi) By Corollary \ref{corollary4} we know that  $\sigma_{gDR}(T)=\sigma_{gDRW}(T)$ if and only if int$\sigma_{*}(T) \subset \sigma_{w}(T),$ where $\sigma_{*}=\sigma_{b}$ or $ \sigma_{ub}$ or  $ \sigma_{lb}$ or $\sigma_{bb}$ or $ \sigma_{lsbb}$ or  $ \sigma_{usbb}$ or $\sigma_{gD}$ or $\sigma_{gD\mathcal{Q}}$ or  $ \sigma_{llsbb}$ or $\sigma_{gDR}$ or $\sigma_{gDR\mathcal{Q}}$ or  $ \sigma_{gDR\mathcal{J}}$ or  $ \sigma_{gDM\mathcal{Q}}$ or $\sigma_{}$. By \cite [Theorem 2.8]{8} we get $\sigma_{gDR}(T)=\sigma_{gDRW}(T)$ if and only if $T$ satisfies  Browder's theorem.
\end{proof}
\begin{theorem}
Let $T \in B(X)$. Then the following statements are equivalent:

(i) generalized a-Browder's theorem holds for $T$ and $\sigma(T)=\sigma_{a}(T)$,

(ii) int$\sigma_{*}(T) \subset \sigma_{usbw}(T),$  where $\sigma_{*}=\sigma_{bb}$ or $ \sigma_{lsbb}$  or $\sigma_{gD}$ or $\sigma_{gD\mathcal{Q}}$  or $\sigma_{gDR}$ or $\sigma_{gDR\mathcal{Q}}$  or  $ \sigma_{gDM}$ or $\sigma_{gD\mathcal{Q}M}$,

(iii) int$\sigma_{*}(T) \subset \sigma_{gDMW_{+}}(T),$  where $\sigma_{*}=\sigma_{bb}$ or $ \sigma_{lsbb}$  or $\sigma_{gD}$ or $\sigma_{gD\mathcal{Q}}$  or $\sigma_{gDR}$ or $\sigma_{gDR\mathcal{Q}}$  or  $ \sigma_{gDM}$ or $\sigma_{gDM\mathcal{Q}}$,

(iv) $\sigma_{*}(T) \subset \sigma_{usbw}(T),$  where $\sigma_{*}=\sigma_{bb}$ or $ \sigma_{lsbb}$  or $\sigma_{gD}$ or $\sigma_{gD\mathcal{Q}}$  or $\sigma_{gDR}$ or $\sigma_{gDR\mathcal{Q}}$ or $ \sigma_{gDM\mathcal{Q}}$ or $\sigma_{gDM}$,

(v) $\sigma_{*}(T) \subset \sigma_{gDMW_{+}}(T),$  where $\sigma_{*}=\sigma_{bb}$ or $ \sigma_{lsbb}$  or $\sigma_{gD}$ or $\sigma_{gD\mathcal{Q}}$  or $\sigma_{gDR}$ or $\sigma_{gDR\mathcal{Q}}$ or  $ \sigma_{gDM\mathcal{Q}}$ or $\sigma_{gDM}$,

(vi) $T^*$ has SVEP at every $\lambda \notin \sigma_{usbw}(T)$, 

(vii) $T^*$ has SVEP at every $\lambda \notin \sigma_{gDMW_{+}}(T).$
\end{theorem}
\begin{proof}
(i) $\Rightarrow$ (ii) Since $\sigma_{usbb}(T)=\sigma_{usbw}(T)$, by \cite[Theorem 3.50]{1} we have acc$\sigma_{a}(T) \subset \sigma_{usbw}(T)$. This gives int$\sigma_{*}(T) \subset$ acc$\sigma(T)=$ acc$\sigma_{a}(T) \subset \sigma_{usbw}(T)$, where $\sigma_{*}=\sigma_{bb}$ or $ \sigma_{lsbb}$ or $\sigma_{gD}$ or $\sigma_{gD\mathcal{Q}}$ or   $\sigma_{gDR}$ or $\sigma_{gDR\mathcal{Q}}$  or  $ \sigma_{gDM\mathcal{Q}}$ or $\sigma_{gDM}$.

(ii) $\Rightarrow$ (i)  Since int$\sigma_{bb}(T)\subset \sigma_{usbw}(T)\subset \sigma_{a}(T)$,  by \cite[Theorem 2.1]{4}  we get $\sigma(T)=\sigma_{a}(T)$. As int $\sigma_{*}(T) \subset \sigma_{usbw}(T),$ where  $\sigma_{*}=\sigma_{bb}$ or $ \sigma_{lsbb}$  or $\sigma_{gD}$ or $\sigma_{gD\mathcal{Q}}$ or $\sigma_{gDR}$ or $\sigma_{gDR\mathcal{Q}}$ or    $ \sigma_{gDM\mathcal{Q}}$ or $\sigma_{gDM}$. This gives $\sigma_{usbw}(T)=\sigma_{bb}(T)$ which implies that $\sigma_{usbw}(T)=\sigma_{usbb}(T)$.

 (ii) $\Rightarrow$ (iii) As int$\sigma_{*}(T) \subset \sigma_{usbw}(T)$, by Corollary \ref{corollary1} we have $\sigma_{bb}(T)= \sigma_{usbw}(T)$. Therefore, by \cite[Theorem 1]{14} and \cite[Theorem 5]{14} we have $\sigma_{gDM}(T)= \sigma_{gDMW_{+}}(T)$. Thus, by Corollary \ref{corollary3} we get  int$\sigma_{*}(T) \subset \sigma_{gDMW_{+}}(T)$.
 
 (iii)$\Rightarrow$ (ii) obvious.
 
 (ii)$\Rightarrow$ (iv) Since int$\sigma_{*}(T) \subset \sigma_{usbw}(T)$, by Corollary \ref{corollary1} we get $\sigma_{bb}(T) = \sigma_{usbw}(T)$. This gives $\sigma_{*}(T) \subset  \sigma_{bb}(T)= \sigma_{usbw}(T)$.
 
 (iv) $\Rightarrow$ (ii) obvious.
 
 (iii) $\Leftrightarrow$ (v)  Same as (ii) $\Leftrightarrow$ (iv).
 
 (iv)  $\Rightarrow$ (vi) Let $\lambda \notin \sigma_{usbw}(T)$. Since $\sigma_{bb}(T) \subset \sigma_{usbw}(T)$, $\lambda \notin \sigma_{bb}(T)$. This gives $q(\lambda I-T) < \infty$ which implies that $T^*$ has SVEP at $\lambda$.
 
 (vi) $\Rightarrow$ (iv) Let $\lambda  \notin \sigma_{usbw}(T)$. Since  $T^*$ has SVEP at $\lambda$, by the proof of Theorem \ref{theorem1} we get $\lambda \notin \sigma_{bb}(T)$. Therefore $\sigma_{*}(T) \subset \sigma_{bb}(T) \subset \sigma_{usbw}(T)$, where  $\sigma_{*}=\sigma_{bb}$ or $ \sigma_{lsbb}$  or $\sigma_{gD}$ or $\sigma_{gD\mathcal{Q}}$  or $\sigma_{gDR}$ or $\sigma_{gDR\mathcal{Q}}$ or $ \sigma_{gDM\mathcal{Q}}$ or $\sigma_{gDM}$.
 
 (vii)  $\Rightarrow$ (vi) Since $\sigma_{gDMW_{+}}(T) \subset \sigma_{usbw}(T)$, $T^{*}$ has SVEP at every $\lambda \notin \sigma_{usbw}(T)$.
 
 (v) $\Rightarrow$ (vii) can be proved same as (iv) $\Rightarrow$ (vi).
\end{proof}
\begin{theorem}
Let $T \in B(X)$. Then the following statements are equivalent:

(i) generalized a-Browder's theorem holds for $T^{*}$ and $\sigma(T)=\sigma_{s}(T)$,

(ii) int$\sigma_{*}(T) \subset \sigma_{lsbw}(T),$  where $\sigma_{*}=\sigma_{bb}$ or $ \sigma_{usbb}$  or $\sigma_{gD}$ or $\sigma_{gD\mathcal{J}\mathcal{Q}}$  or $\sigma_{gDR}$ or $\sigma_{gDR\mathcal{J}}$  or  $ \sigma_{gDM}$ or $\sigma_{gDM\mathcal{J}}$,

(iii) int$\sigma_{*}(T) \subset \sigma_{gDMW_{-}}(T),$  where $\sigma_{*}=\sigma_{bb}$ or $ \sigma_{usbb}$  or $\sigma_{gD}$ or $\sigma_{gD\mathcal{J}}$  or $\sigma_{gDR}$ or $\sigma_{gDR\mathcal{J}}$  or  $ \sigma_{gDM}$ or $\sigma_{gDM\mathcal{J}}$,

(iv) $\sigma_{*}(T) \subset \sigma_{lsbw}(T),$   where $\sigma_{*}=\sigma_{bb}$ or $ \sigma_{usbb}$  or $\sigma_{gD}$ or $\sigma_{gD\mathcal{J}}$  or $\sigma_{gDR}$ or $\sigma_{gDR\mathcal{J}}$  or  $ \sigma_{gDM}$ or $\sigma_{gDM\mathcal{J}}$,

(v) $\sigma_{*}(T) \subset \sigma_{gDMW_{-}}(T),$  where $\sigma_{*}=\sigma_{bb}$ or $ \sigma_{usbb}$  or $\sigma_{gD}$ or $\sigma_{gD\mathcal{J}}$  or $\sigma_{gDR}$ or $\sigma_{gDR\mathcal{J}}$  or  $ \sigma_{gDM}$ or $\sigma_{gDM\mathcal{J}}$,

(vi) $T$ has SVEP at every $\lambda \notin \sigma_{lsbw}(T)$,

(vii) $T$ has SVEP at every $\lambda \notin \sigma_{gDMW_{-}}(T).$
\end{theorem}
\begin{proof}
(i) $\Rightarrow$ (ii) Since $T^*$ satisfies generalized a-Browder's theorem, by \cite[Theorem 2.2]{16} $T^*$ satisfies a-Browder's theorem. This implies that $\sigma_{ub}(T^*)=\sigma_{uw}(T^*)$. Therefore, $\sigma_{lb}(T)=\sigma_{lw}(T)$. Let $\lambda_0 \notin \sigma_{lsbw}(T)$. Then by \cite[Theorem 1.117]{1} there exists an open disc $D$ centred at $\lambda_0$ such that $\lambda I-T$ is lower semi-Weyl for all $\lambda \in D\setminus \{\lambda_0\}$. This gives $\lambda I-T$ is lower semi-Browder for all $\lambda \in D \setminus \{\lambda_0\}.$ Therefore, $T^*$ has SVEP at every $\lambda \in D \setminus \{\lambda_0\}.$ Thus, $T^*$ has SVEP at $\lambda_0$. This gives $\lambda_0 \notin \sigma_{lsbb}(T).$ Hence, $\sigma_{lsbw}(T)=\sigma_{lsbb}(T).$ By \cite[Theorem 3.50]{1} we get acc$\sigma_{s}(T) \subset \sigma_{lsbw}(T)$. This gives int$\sigma_{*}(T) \subset$ acc$\sigma(T)=$ acc$\sigma_{s}(T) \subset \sigma_{lsbw}(T)$, where $\sigma_{*}=\sigma_{bb}$ or $ \sigma_{usbb}$  or $\sigma_{gD}$ or $\sigma_{gD\mathcal{J}\mathcal{Q}}$  or $\sigma_{gDR}$ or $\sigma_{gDR\mathcal{J}}$  or  $ \sigma_{gDM}$ or $\sigma_{gDM\mathcal{J}}$.

(ii) $\Rightarrow$ (i)  Since int$\sigma_{bb}(T)\subset \sigma_{lsbw}(T)\subset \sigma_{s}(T)$,  by \cite[Theorem 2.1]{4}  we get $\sigma(T)=\sigma_{s}(T)$. As int$\sigma_{*}(T) \subset \sigma_{lsbw}(T),$ where $\sigma_{*}=\sigma_{bb}$ or $ \sigma_{usbb}$  or $\sigma_{gD}$ or $\sigma_{gD\mathcal{J}\mathcal{Q}}$  or $\sigma_{gDR}$ or $\sigma_{gDR\mathcal{J}}$  or  $ \sigma_{gDM}$ or $\sigma_{gDM\mathcal{J}}$. This gives $\sigma_{lsbw}(T)=\sigma_{bb}(T)$ which implies that $\sigma_{lsbw}(T)=\sigma_{lsbb}(T)$.

 (ii) $\Rightarrow$ (iii) As int$\sigma_{*}(T) \subset \sigma_{lsbw}(T)$, by Corollary \ref{corollary2} we have $\sigma_{bb}(T)= \sigma_{lsbw}(T)$. Therefore, by \cite[Theorem 1]{14} and \cite[Theorem 5]{14} we have $\sigma_{gDM}(T)= \sigma_{gDMW_{-}}(T)$. Thus, by Corollary \ref{corollary3} we get  int$\sigma_{*}(T) \subset \sigma_{gDMW_{+}}(T)$.
 
 (iii)$\Rightarrow$ (ii) obvious.
 
 (ii)$\Rightarrow$ (iv) Since int$\sigma_{*}(T) \subset \sigma_{usbw}(T)$, by Corollary \ref{corollary1} we get $\sigma_{bb}(T) = \sigma_{usbw}(T)$. This gives $\sigma_{*}(T) \subset  \sigma_{bb}(T)= \sigma_{usbw}(T)$.\\
 (iv) $\Rightarrow$ (ii) obvious.
 
 (iii) $\Leftrightarrow$ (v)  Same as (ii) $\Leftrightarrow$ (iv).
 
 (iv)  $\Rightarrow$ (vi) Let $\lambda \notin \sigma_{lsbw}(T)$. Since $\sigma_{bb}(T) \subset \sigma_{lsbw}(T)$, $\lambda \notin \sigma_{bb}(T)$. This gives $p(\lambda I-T) < \infty$ which implies that $T$ has SVEP at $\lambda$.\\
 (vi) $\Rightarrow$ (iv) Let $\lambda  \notin \sigma_{lsbw}(T)$. Since  $T$ has SVEP at $\lambda$, by the proof of Theorem \ref{theorem1} we get $\lambda \notin \sigma_{bb}(T)$. Therefore, $\sigma_{*}(T) \subset \sigma_{bb}(T) \subset \sigma_{lsbw}(T)$, where $\sigma_{*}=\sigma_{bb}$ or $ \sigma_{usbb}$  or $\sigma_{gD}$ or $\sigma_{gD\mathcal{J}}$  or $\sigma_{gDR}$ or $\sigma_{gDR\mathcal{J}}$  or  $ \sigma_{gDM}$ or $\sigma_{gDM\mathcal{J}}$.
 
 (vii)  $\Rightarrow$ (vi) Since $\sigma_{gDMW_{-}}(T) \subset \sigma_{lsbw}(T)$, $T$ has SVEP at every $\lambda \notin \sigma_{lsbw}(T)$.
 
 (v) $\Rightarrow$ (vii) same as (iv) $\Rightarrow$ (vi).
\end{proof}

\end{document}